\documentclass{article}
\usepackage{geometry}                
\usepackage{graphicx, subfigure}
\usepackage{amssymb, amsmath, amsfonts, amscd}
\usepackage{epstopdf}
\usepackage{algorithm}
\usepackage{algorithmic}
\usepackage{multirow}
\usepackage{color}
\newcommand{\ord}{\operatorname{ord}}
\newenvironment{proof}[1][Proof]{\noindent\textbf{#1.} }{\ \rule{0.5em}{0.5em}}
\newtheorem{theorem}{Theorem}

\newtheorem{conjecture}{Conjecture}
\newtheorem{definition}{Definition}
\newtheorem{example}{Example}[section]

\newtheorem{remark}{Remark}

\newcommand{\bbC}{\mathbb{C}}
\newcommand{\bbD}{\mathbb{D}}
\newcommand{\bbN}{\mathbb{N}}

\newcommand{\bbR}{\mathbb{R}}

\newcommand{\bfa}{\mathbf{a}}
\newcommand{\bfb}{\mathbf{b}}
\newcommand{\bfc}{\mathbf{c}}
\newcommand{\bff}{\mathbf{f}}
\newcommand{\bfr}{\mathbf{r}}
\newcommand{\bfx}{\mathbf{x}}
\newcommand{\bfy}{\mathbf{y}}
\newcommand{\bfz}{\mathbf{z}}
\newcommand{\bfF}{\mathbf{F}}
\newcommand{\cA}{\mathcal{A}}
\newcommand{\cB}{\mathcal{B}}
\newcommand{\cG}{\mathcal{G}}
\newcommand{\cM}{\mathcal{M}}

\newcommand{\cO}{\mathcal{O}}
\newcommand{\cP}{\mathcal{P}}

\newcommand{\Ffrak}{\mathfrak{F}}

\newcommand{\Rfrak}{\mathfrak{R}}
\newcommand{\argmin}{\operatornamewithlimits{arg \, min}}
\newcommand{\argmax}{\operatornamewithlimits{arg \, max}}
\newcommand\pa{\partial}

\title{An Enhancement Algorithm for Cyclic Adaptive Fourier Decomposition}
\author{
Tao Qian
\thanks{Department of Mathematics, University of Macau, Macao, e-mail: fsttq@umac.mo} \ and \  
Jianzhong Wang
\thanks{Department of Mathematics and Statistics, Sam Houston State University, Huntsville, TX 77341, e-mail: jzwang@shsu.edu}\\
}
\date{April 30, 2018}

\begin{document}
\maketitle
\begin{abstract}
	One important problem in the theory of Hardy space is to find the best rational approximation of a given order to a function in the Hardy space $H^2$ on the unit disk. It is equivalent to finding the best Blaschke form with free poles. The cyclic adaptive Fourier decomposition method is based on the grid search technique. Its approximative precision is limited by the grid spacing. This paper propose an enhancement of the cyclic adaptive Fourier decomposition. The new method first changes the rectangular grids to the polar ones for cyclic adaptive Fourier decomposition, so that the decomposition enables us to employ fast Fourier transformation to accelerate the search speed. Furthermore, the proposed algorithm utilizes the gradient descent optimization to tune the best pole-tuple on the mesh grids, reaching higher precision with less computation cost. Its validity and effectiveness are confirmed by several examples.
\end{abstract}



\section{Introduction}
Throughout the paper, we denote by $\bbD$ the open unit disc, and $H^2=H^2({\bbD})$ the Hardy $H^2$-space on ${\bbD}$:
\[ H^2 =\{ f(z)=\sum_{k=0}^\infty c_kz^k\ :\ \sum_{k=0}^\infty |c_k|^2<\infty\},\]
which is equipped with the inner product
$$ \langle f,g\rangle = \frac{1}{2\pi i}\oint_{\partial\bbD} f(z)\overline{zg(z)}\, dz = \frac{1}{2\pi i}\oint_{\partial\bbD}\frac {f(z)\overline{g(z)}}{z}\, dz,\quad f,g\in H^2. $$

Now we recall the classical definition of \textit{best n-rational approximation}. Let $p$ and $q$ be polynomials, and all zeros of $q$ be outside the closed unit disc. In this paper, we always assume that $p$ and $q$ are coprime so that all rational functions are in the form of $p/q$ is \textit{non-degenerate}. The \textit{order of a rational function} $p/q$ is defined by $\ord(p/q)=\max \{{\rm deg }(p), {\rm deg }(q)\}.$
A \textit{best $n$-rational approximation} to $f\in H^2$ is an $n$-order rational function $p_1/q_1$ that satisfies
\begin{equation}\label{rational best}
\| f-p_1/q_1\| \leq \| f- p/q\|,\quad \ord(p/q)\le n.
\end{equation}
An important type of rational approximation is Blaschke form approximation, which is briefly introduced in the following:
For a given  $n$-vector $\bfa=[a_1,\cdots,a_n]^T\in \bbD^n$, the $n$-Takenaka-Malmquist orthonormal rational function system $\{B_k\}_{k=1}^n$ is defined by
\begin{equation}\label{eq1}
B_k(z)=\frac{\sqrt{1-|a_k|^2}}{1-\bar{a}_k z}\prod_{j=1}^{k-1}\frac{z-a_j}{1-\bar{a}_jz}.
\end{equation}
In this paper, we assume that all of the components of the $n$-vector $\bfa$ are distinct.
We denote by $L(\bfa)$ the linear subspace of $H^2$ spanned by $\{B_k\}_{k=1}^n$, and call a function in $L(\bfa)$ being of the \emph{$n$-Blaschke form}. Then the orthogonal projection of $f\in H^2$ to $L(\bfa)$ is the $n$-Blaschke form
\begin{equation}\label{neww1}
f_n=\sum_{k=1}^n \langle f,B_k\rangle B_k.
\end{equation}
It is known that the linear space $L(\bfa)$ is invariant under the permutations of $\bfa$. Hence, we often identify $L(\bfa)$ with $L(\cA)$,  where $\cA$ is the $n$-tuple $\{a_1,a_2,\cdots,a_n\}$.
We define the squared $H^2$-error of the projection (\ref{neww1}) by
\begin{equation}\label{herr}
 A(f;\cA)=\|f-f_n\|^2=\|f\|^2-\sum_{k=1}^{n}|\langle f,B_k\rangle|^2,
 \end{equation}
and call $n$ \emph{the approximation degree}, and $E(f,\cA)=\sum_{k=1}^{n}|\langle f,B_k\rangle|^2$  \emph{the energy of $f$} (at $\cA$).
We say that an $n$-tuple $\cB=\cB(n)=\{b_1,b_2,\cdots,b_n\}$ induces the \emph{$n$-best Blaschke form approximation} to $f$ if and only if
\begin{equation}\label{eq22}
\cB=\argmin_{\cA\subset \bbD} A(f,\cA),
\end{equation}
or equivalently,
\begin{equation}\label{eq222}
\cB=\argmax_{\cA\subset \bbD} E(f,\cA).
\end{equation}
Later, $\cB$ is called the \emph{best $n$-tuple} (for the approximation). The relation between $n$-Blachke form approximation and $n$-order rational one was specified in \cite{Qian13}. The methods and algorithms for both of them are quite similar.   Hence, in this paper we mainly focus on the $n$-best Blachke form approximation. More detailed discussions on the properties of $n$-best Blachke form approximations and the corresponding best $n$-tuples can be found in \cite{MQW, QianW11, QianW13, QianZL11, Qian13}.

Rational approximation is of a great significance in both pure and applied mathematics. As an example, in system identification, one wishes to approximate the system function by rational ones.

Although the study of the $n$-best rational approximation has a long history \cite{Ba, Wa, MQW, QianW13}, practical algorithms for finding the approximation are still under research. In literature, Baratchart's group in \cite{BCO,FM} proposed the method  based on the second
derivative test, treating the coefficients of the polynomial $q$ as the parameters. Qian's group proposed the adaptive Fourier decomposition algorithm (AFD) \cite{QianZL11} and its improvement \emph{cyclic AFD} (shortly, CAFD) \cite{Qian13}, which used the poles of the approximative rational function as parameters. They created a search scheme to find the best tuple over the rectangular grids (\cite{BCO,DMA}).

In a pole-tuple search algorithm, the time cost is a main issue. If a 1-D rectangular $\epsilon$-net $\cM\subset\bbD$ has $N\times M$ nodes, an exhaustive search on the $n$-dimensional net $\cM^n$ needs $O(N^{2n}M^{n})$ times of operations, which is unpractical when the approximation degree $n$ is high. To reduce the computational cost, Qian introduced \emph{coordinate maximum} \cite[Definition 1]{Qian13} and proved that it is identical with the best $n$-tuple if the target function satisfies a certain condition. Based on this fact, AFD suggests $n$ rounds of coordinate-by-coordinate search on $\cM$. In each round, it finds the maximum only for a variable over the 1-D mesh:
\begin{equation*}
\hat{b}_j=\argmax_{a_j\subset \cM} E(f,\cA)\quad 1\le j\le n.
\end{equation*}
Therefore, AFD needs only $O(N^2M)$ operations, dramatically reducing the time used in an exhaustive search.

Unfortunately, AFD usually provides a notable error. CAFD improves AFD by running several cycles of AFD till the coordinate maximum on $\cM^n$ is obtained.

In CAFD, the accuracy of the best $n$-tuple approximation is limited by the grid gap $\epsilon$. To improve the approximation, we must choose a smaller $\epsilon$, which increases the computational time significantly.

To overcome the limitation of grids search technique and to reduce the time cost, in this paper, we propose a new algorithm CAFD-CGD to enhance CAFD. The proposed algorithm consists of two phases: Firstly, we utilize the polar-type mesh grids in CAFD, and employ the best $n$-tuple on the mesh as initial $n$-tuple for the second phase. Secondly, we launch a gradient decent optimization in a neighborhood of the initial $n$-tuple to tune the best $n$-tuple. The new algorithm has two advantages: First, the polar-grid set enables us to employ the fast Fourier transformation (FFT), reducing the search time from $O(N^2M)$ to $O(NM \log N)$. Second, theoretically, the complex gradient descent method (CGD) can exactly locate the best $n$ tuple in $\mathbb{D}^n$, breaking through the limitation of the grid gap.

The paper is organized as follows: In Section~\ref{GDM}, we develop the \emph{complex gradient decent algorithm} (CGD) to tune the initial best $n$-tuple found on a mesh grid set. In Section~\ref{MSP}, we first study the uniqueness of the best $n$-tuple. Then we introduce the  polar-type mesh grids and apply CAFD to find the best $n$-tuple on the grids.   In Section~\ref{feval}, utilizing FFT, we develop the novel \emph{fast evaluation algorithm}(FEVAL) for the case that a function $f\in H^2$ is given by its samples on the unit circle, and show that the new search scheme reduces cost of CAFD from $O(N^2M)$ to $O(NM\log N)$. In the last section, we give several illustrative examples to show the effectiveness and accuracy of our algorithm.

\section{Complex Gradient Descent Algorithm (CGD)}\label{GDM}
We first introduce some notions and notations. For $\bfz=[z_1, \cdots, z_n]^T\in \bbC^n$, its conjugate is denoted by $\bar{\bfz}=[\bar{z}_1,\cdots,\bar{z}_n]^T\in\bbC^n$.  Write $\bfz=\bfx +i\bfy$, $\bfr=[\bfx,\bfy]$, and $\bfc=[\bfz,\bar{\bfz}]$. Hence, a complex function $f(\bfz):\bbC^n\to \bbC,$ with a little abuse of notation, has the following different forms:
\[f(\bfz)=f(\bfz,\bar{\bfz})=f(\bfc)=f(\bfx,\bfy)=f(\bfr).
\]
As usual, we define the cogradient operator by $\frac{\pa}{\pa \bfz}=\left[\frac{\pa}{\pa {z}_1},\frac{\pa}{\pa {z}_2}, \cdots, \frac{\pa}{\pa {z}_n}\right]$, the conjugate cogradient operator by $\frac{\pa}{\pa \bar{\bfz}}=\left[\frac{\pa}{\pa \bar{z}_1},\frac{\pa}{\pa \bar{z}_2}, \cdots, \frac{\pa}{\pa \bar{z}_n}\right]$,
and the gradient of a differentiable function $f(\bfz,\bar{\bfz})$ by $\nabla_\bfz f=\left(\frac{\partial f}{\partial \bfz}\right)^H$, where $(\cdot)^H$ denotes the Hermitian transpose. We also denote  by $\langle\bfa,\bfb\rangle=\bfa^H\bfb$ the inner product of two complex $n$-vectors $\bfa,\bfb\in \bbC^n$.

We now assume that $\bfa\in\bbC^n$ is a local minimal-value point of a real-valued function $g(\bfz)$. The complex gradient descent method for finding  $\bfa$ is the following:  Let $\bfa_0$ be the initial guess, who resides in a neighborhood of $\bfa$. We find $\bfa$ as the limit of the sequence of $(\bfa_k)$:
\begin{equation}\label{eq3c}
\bfa_{k+1}=\bfa_k- t_k\nabla g(\bfa_k),\quad k=0,1,2,\cdots,
\end{equation}
where $\nabla g$ in (\ref{eq3c}) is Lipschitz continuous with constant $L>0$, i.e.,
\begin{equation}\label{eqLip}
    \|\nabla g(\bfa)-\nabla g(\bfb)\|\le L\|\bfa-\bfb\|.
\end{equation}

In this paper, we adopt \emph{backtracking line search}, in which a fixed  $\beta, 0<\beta<1,$ is employed for formulating $t_k$ by $t_k=\beta t_{k-1}, t_1=1$. Then the iteration is terminated when
\begin{equation}\label{bline}
g(\bfa_k-t_{k}\nabla g(\bfa_k) )>g(\bfa_k)-\frac{t_k}{2}\|\nabla g(\bfa_k)\|^2.
\end{equation}
We call the method above \emph{Complex Gradient Descent Method (CGD)}.
The convergence theorems of the real gradient descent methods in \cite{BoydV04, HastieTF09, Nesterov04} can also be applied for CGD. Because

In our problem, we will set $g=-E$ in (\ref{eq3c}), where $E(\bfa)$ is the energy function of $f$. Besides, to guarantee that after each iterative step the new $n$-tuple $\bfa_k$ is still in $\bbD^n$, $t_k$ in (\ref{bline}) must also satisfy
$$\bfa_k+t_k\nabla E(\bfa_k)\in \mathcal{N}_{\bfa_k}\cap \bbD^n,$$
where
$\mathcal{N}_{\bfa_k}=\left\{\bfz; \|(a_k)_j-z_j\|<r,\quad 1\le j\le n \right\}.$

An effective formulation for the gradient $-\nabla E(=\nabla A)$ is the key step in CGD. We now establish such a formulation as follows. Write $$e_a(z)=\frac{\sqrt{1-|a|^2}}{1-\bar{a}z}, \quad a\in\bbC\setminus\{0\},
$$
and let $P_\ell$ be the permutation of the index set $\{1,2,\cdots,n\}$ such that $P_\ell(n)=\ell$. For a given analytic function $f\in H^2$,  we inductively define $n$ functions $f_{P_\ell(j)}, 1\le j\le n,$ by
\begin{eqnarray}
f_{P_\ell(1)}(z) & = & f(z),\nonumber \\
f_{P_\ell(j)}(z) & = & \frac{1-z\bar{a}_{P_\ell(j-1)}}{z-a_{P_\ell(j-1)}} \left(f_{P_\ell(j-1)} (z)-\langle f_{P_\ell(j-1)},e_{a_{P_\ell(j-1)}}\rangle e_{a_{P_\ell(j-1)}}(z)\right).  \label{fhatj}
\end{eqnarray}
It was proved that all $f_{P_\ell(j)}(z)$ are analytic in $\bbD$ \cite{QianW11}.
Since the energy function $E(\bfa)$ is invariant under the permutation, it has $n$ different representations:
\begin{equation}\label{ferr1}
E(\bfa)=\sum_{j=1}^n \left( 1-|a_{P_\ell(j)}|^2 \right) \left|f_{P_\ell(j)} (a_{P_\ell(j)})\right|^2, \quad  \ell= 1, \cdots, n.
\end{equation}
Note that the variable $a_\ell = a_{P_\ell(n)}$ only occurs in the last term of the sum in (\ref{ferr1}). Since $f_{P_\ell(n)}$ is analytic,  we have $\frac{\overline{f_{P_\ell(n)}}}{\partial z_\ell}=0$, so that
\begin{equation}\label{dAk}
\frac{\partial(-E(\bfa))}{\partial z_\ell}=\overline{f_{P_\ell(n)}(a_\ell)}\left(\overline{a_\ell}f_{P_\ell(n)}(a_\ell)-(1-|a_\ell|^2)f^\prime_{P_\ell(n)}(a_\ell)\right),\quad \ell=1,\cdots, n,
\end{equation}
where $f^\prime_{P_\ell(n)}$ can be computed by the following recursive formula:
\begin{align*}
f^\prime_{P_\ell(1)}(z) & =f^\prime(z),\\
   f^\prime_{P_\ell(j)}(z) & =f^\prime_{P_\ell(j-1)}(z)\frac{1-\overline{a_{P_\ell(j-1)}}z}{z-a_{P_\ell(j-1)}} + \left(f_{P_\ell(j-1)}(z)-f_{P_\ell(j-1)}(a_{P_\ell(j-1)})\right)\frac{|a_{P_\ell(j-1)}|^2-1}{(z-a_{P_\ell(j-1)})^2}.
\end{align*}

\begin{remark}
The permutation $P_\ell$ above is not unique. In practice, we adopt $P_\ell=P^\ell$, where $P$ is the 1-\emph{shift} permutation: $P(1,2,\cdots,n)=(2,\cdots,n,1)$.
\end{remark}

We now present the pseudo-code of  CGD in \textbf{Algorithm~\ref{alg1}}.
Its inputs consist of the function $f\in H^2$, an initial $n$-tuple $\bfa$, which is near the best $n$-tuple $\bfb$, a tolerance $\varepsilon$, and the parameter $\beta$ in the backtracking line search. Here, the initial $n$-tuple $\bfa$ is found by \textbf{Algorithm~\ref{alg2} (ITS)} in the next section. When the tolerance condition $\|\nabla E(\bfb)\|^2<\varepsilon$ holds, we terminate the algorithm and output the $n$-tuple $\bfb$.

\begin{algorithm}[htbp]\caption{ \textbf{CGD}: Complex gradient descent algorithm for finding the best tuple} \label{alg1}
     \begin{algorithmic}[1]

    \REQUIRE
    $\bfa,f,\beta$, neighbor size $r$, and the tolerance $\varepsilon$.

    \STATE Compute $\nabla E(\bfa)$ using $f$ and $\bfa$.
    \WHILE{$\|\nabla E(\bfa)\|^2 >\varepsilon$}
        \STATE Find $s_1>0, s_2>0$ such that $\|\bfa + s_1 \nabla E(\bfa)\|_\infty=1$ and $\|s_2 \nabla E(\bfa)\|_\infty=r$. Set $s=min(s_1,s_2)$.
        \STATE Compute $\bfc = \bfa + s \nabla E(\bfa)$.
        \WHILE{$E(\bfc)< E(\bfa)+\frac{s}{2}\|\nabla E(\bfa))\|^2$}
            \STATE Update $s$:  $s = \beta s$.
            \STATE Update $\bfc$: $\bfc = \bfa + s \nabla E(\bfa)$.
        \ENDWHILE
        \STATE Update $\bfa$: $\bfa =\bfc$.
        \STATE Re-compute  $\nabla E(\bfa)$.
    \ENDWHILE
    \STATE Set the output: $\bfb =\bfa$.
    \ENSURE  $\bfb$
     \end{algorithmic}
\end{algorithm}

\section{CAFD on a Polar Mesh Grid}\label{MSP}
A main difficulty for best $n$-tuple search algorithms is that the energy function $E(\bfa)$ lacks the uniqueness of the global maximum (or, equivalently, the error function $A(\mathbf{a})$ lacks the uniqueness of the global minimum). Therefore, a coordinate maximum of $E(\bfa)$ is not necessarily a global one. Besides, lack of the uniqueness of the global minimum discourages the approach of globally convex optimization.
\begin{example}
Let $f(z)=z^k, k\in\bbN$. We consider its Blaschke-form approximation of degree $1$. By (\ref{ferr1}), the error function $A(\bfa)$ has the form of
\[ \|f\|^2-(1-|a|^2)|f(a)|^2=1-(1-|a|^2)|a|^{2k},
\]
which reaches the global minimum when $|a|=\sqrt{\frac{k}{k+1}}$.
\end{example}
The example shows that, when the approximation degree is one, the global minimum may occur on a manifold.

When the approximation degree $n>1$, due to the invariance of the energy function $E(\bfz)$ under the permutations of $\bfz$,  $E(\bfz)$ reaches its global maximum at least at $n!$ distinct points in $\bbD^n$, of which each is a permutation of another one. This fact totally denies the uniqueness of the $n$-best tuple. When $n>1$, we have the following result on the set of best $n$-tuples.
\begin{theorem}\label{th3.1}
When $n>1$, the set of best $n$-tuples of the Blaschke-form approximation of degree $n$ for a function $f\in H^2$ cannot contain a continuous curve in $\bbD^n$.
\end{theorem}
\begin{proof} Denote by $S$ the set of the best $n$-tuples of the approximation. Then $\nabla E(\bfa)=0$ for any $\bfa\in S$.  By (\ref{dAk}), $\frac{\partial E(\bfa)}{\partial z_k}=0$ if and only if either $f_{P_k(n)}(a_k)=0$, or
\begin{equation}\label{df0}
\overline{a_k}f_{P_k(n)}(a_k)-(1-|a_k|^2)f^\prime_{P_k(n)}(a_k)=0.
\end{equation}
By (\ref{ferr1}), $E(\bfa)$ does not have the maximum value when $f_{P_k(n)}(a_k)=0$. Therefore, if $\bfa\in S$, \eqref{df0} must hold for each $k$.

Assume there is a continuous curve $\Gamma\subset S$. Denote by $\Gamma_k$ the projection of $\Gamma$ on the $k$-th coordinate (complex) plane. Then at least one $\Gamma_k$ is a continuous curve in $\bbD$. Assume $k=n$  and $P_n=I$. Then by (\ref{df0}),
\begin{equation*}
\overline{a_n}f_{n}(a_n)-(1-|a_n|^2)f^\prime_{n}(a_n)=0, \quad  \forall a_n\in \Gamma_n.
\end{equation*}
It follows that
\begin{equation}\label{dfn0}
\frac{z f^\prime_{n}(z)}{f_{n}(z)}=\frac{|z|^2}{1-|z|^2}, \quad \forall z\in \Gamma_n.
\end{equation}
Hence, $g(z)=\frac{z f^\prime_{n}(z)}{f_{n}(z)}$ is real-valued and has no poles on $\Gamma_n$. Therefore, $g(z)$ is analytic and $g(z)=c$ on $\cO$, where $\cO$ is a connected domain with $\Gamma_n\subset \cO$. It follows that $f_n(z)=rz^c$ on $\bbD$. However, by the formula (\ref{fhatj}), $f_n(z)$ has at least $n$ distinct zeros in $\bbD$, which leads to a contradiction. The proof is completed.
\end{proof}
\smallskip

According to the theorem, we make the following conjecture:
According to Theorem~\ref{th3.1}, we make the following conjecture:
\begin{conjecture} When $n>1$, the energy function $E(\bfz)$ in (\ref{ferr1}) has exact $n!$ best $n$-tuples, i.e., all of its best $n$-tuples are permutations of a single one in $\mathbb{D}^n$.
\end{conjecture}
Assume the conjecture holds. By the similar argument given in the proof of \cite[Corollary 4]{Qian13}, we can confirm that a coordinate maximum point for $E(\bfz)$ is also a global maximum one. In this case, on given grids, we can employ CADF to find an $n$-tuple, which is nearest to a best $n$-tuple.

Considering the geometric structure of the unit disk, we suggest making the CAFD search over a \emph{polar grid set}.
\begin{definition}\label{def:3.1}
Let $M>1$ and $N>1$ be two positive integers, and $\epsilon=\frac{1}{M}, \delta=\frac{1}{N}$. A set of polar $\epsilon$-$\delta$ grids on $\bbD$ is the node set
\begin{equation}\label{meshg}
\cG = \{ z; \ z= m\epsilon e^{2n\delta\pi i},\ 1\le m< M, 1\le n\le N\}.
\end{equation}
\end{definition}
Note that the role of CAFD search here is to find an initial $n$-tuple for CAFD-CGD algorithm. To distinguish it from the standard CAFD in \cite{Qian13}, we will call it the \emph{Initial Tuple Selection Algorithm} (ITS),

We introduce the polar grid set because it enables us to adopt Fast Fourier Transform (FFT) to compute $\langle f,e_a\rangle$ over the grid set $\cG$, when $f$ is given as a digital signal. We will discuss the evaluation of $\langle f,e_a\rangle$ in the next section.

In ITS, the search starts from a randomly chosen $n$-vector $\bfa\in \cG^n$. Fixing $a_1,\cdots,a_{n-1}$, the algorithm first finds the maximal-value point $\tilde{a}_n$ for $|\langle f_n, e_z\rangle|$ over the grid set $\cG$. Then, after $a_n$ is replaced by $\tilde{a}_n$ and $\bfa$ is permuted by the 1-shift permutation, the search process above will be repeated till no replacement can be made.

We present ITS in \textbf{Algorithm~\ref{alg2}}. The inputs of the algorithm consist of a function $f\in H^2$, a randomly selected $n$-tuple $\bfa\in \bbD^n$, two parameters $\epsilon>0$ and $\delta>0$ for polar grid set $\cG$, and a tolerance $\eta>0$. It outputs an $n$-tuple.

\begin{algorithm}[htbp]\caption{\textbf{ITS}: Searching initial tuple for CAFD-CGD algorithm} \label{alg2}
     \begin{algorithmic}[1]
    \REQUIRE
    $\bfa, f, \epsilon, \delta,$ and the tolerance $\eta$.
    \STATE Create the polar-type ($\epsilon, \delta$)-grid set $\cG$ on $\bbD$.
    \STATE Randomly select $\bfa=[a_1,\cdots,a_n] \in \bbD^n$ as the starting $n$-tuple for the algorithm.
    \STATE Create the function $f_n$ and compute the partial energy  $V=|\langle f_n, e_{a_n} \rangle|$.
    \STATE Initialize parameter for WHILE loop and set the WHILE-LOOP light $s=1$.
    \WHILE{$s\ne 0$}
        \STATE Reset $s=0$.
        \FOR{$j=1\to n$}
            \STATE Compute $V_t=\max_{z\in\cG} |\langle f_n, e_{z} \rangle|, a_t = \argmax_{z\in\cG} |\langle f_n, e_{z} \rangle|$.
            \IF{$V_t > V + \eta$}
                \STATE Update:  $a_n = a_t, V=V_t, s=s+1$.
            \ENDIF
            \STATE Permute $\bfa$ using $\bfa = P\bfa=[a_n, a_1,\cdots, a_{n-1}]$.
            \STATE Update $f_n$ based on the new $\bfa$.
            \STATE Update $V=|\langle f_n, e_{a_n} \rangle|$.
        \ENDFOR
    \ENDWHILE
    \ENSURE  $\bfa$
     \end{algorithmic}
\end{algorithm}

\section{Fast Evaluation Algorithm (FEVAL)} \label{feval}
Evaluating the inner products $\langle f_n, e_{z} \rangle$ for all $z$ in a mesh grid set costs most time in a search algorithm. For instance, if we make the CAFD search over an $N\times M$ rectangular mesh grid set $\cM$, we need $O(N^2M)$ operations for all evaluations. The polar grid encourages a fast evaluations of all $\langle f_n, e_{z} \rangle$  due to the following theorem.
Define
\begin{equation}\label{fFourn}
\hat{f}(n)=\frac{1}{2\pi}\int_{0}^{2\pi} f(e^{it})e^{-int}\, dt,\quad f\in H^2.	
\end{equation}

\begin{theorem}\label{th:3.2}
Let $z=re^{it}, 0<r<1$, and $f\in H^2$. Then, the inner product $\langle f, e_z\rangle$ has the following representation:
\begin{equation}\label{fe}
\langle f,e_z\rangle =\sqrt{1-r^2}\sum_{k=0}^{\infty}r^k\hat{f}(k)e^{ikt},
\end{equation}
which yields
\begin{equation}\label{fate}
f(z)(=f(re^{it}))==\sum_{k=0}^{\infty}r^k\hat{f}(k)e^{ikt},
\end{equation}
\end{theorem}
\begin{proof}
For $0<r<1$, we write $p_r(z) = \frac{1}{1-rz}$. It is clear that $\hat{p_r}(k)=r^k$. Then
\begin{equation*}
\langle f, e_z\rangle =\frac{1}{2\pi}\int_{0}^{2\pi}\frac{\sqrt{1-r^2}f(e^{i\tau})}{1-re^{it}e^{-i\tau}}\, d\tau =\sqrt{1-r^2}\sum_{k=0}^{\infty}\hat{f}(k)\hat{p_r}(k)e^{ikt},
\end{equation*}
which yields (\ref{fe}). By $\langle f, e_z\rangle=\sqrt{1-r^2}f(z)$, we have (\ref{fate}).
\end{proof}
\smallskip

For an infinite sequence $\bfc=(c_0, c_1, \cdots, c_k, \cdots)\in\ell^2$, we define the scaling operator $\Rfrak_r:\ell^2\to\ell^2$ as follows: Let $\mathbf{d} = \Rfrak_r (\bfc)$. Then
\begin{equation}\label{Rr}
d_k = r^k c_k, \quad k=0,1,2, \cdots.
\end{equation}
We also denote the discrete Fourier transform by $\Ffrak: H^2\to \ell^2$:
\[ \Ffrak (f(e^{it})) = \hat(f)\equiv \left(\hat{t}(k)\right)_{k=0}^\infty.\]
Then its inverse $\Ffrak^{-1}$ has the form of
\[ \Ffrak^{-1}(\hat{f}) = \sum_{k=0}^\infty \hat{f}(k) e^{ikt} (=f(e^{it}).\]
Thus, we can re-write the formula \eqref{fate} as
\[ f(re^{it})= \Ffrak^{-1}\circ\Rfrak_r\circ\Ffrak (f(e^{it})).\]

Based on Theorem~\ref{th:3.2}, we develop the algorithm (FEVAL) for fast evaluating
$\langle f, e_z\rangle$ for all $z\in \cG$. It only needs $O(NM\log N)$ operations.

\begin{algorithm}[htbp]\caption{\textbf{FEVAL}: Evaluating $\langle f,e_z\rangle$ over the mesh grid set $\cG$} \label{alg3}
     \begin{algorithmic}[1]
    \REQUIRE
    $M , N, \epsilon$ and $\bff$.
    \STATE Initialization: Compute $\hat{f}^{(0)}= \mathrm{FFT}(\bff)$ and scale it to $\hat{f}^{(1)}=\Rfrak_\epsilon(\hat{f}^{(0)})$.
    \STATE Output the first row of $\bfF_e$: $F(1,:) =\sqrt{1-\epsilon^2} \mathrm{IFFT}(\hat{f}^{(1)})$.
    \FOR{ m = 2 : M}
        \STATE Compute the Fourier coefficient sequence on the $m$-th circle: $\hat{f}^{(m)}= \Rfrak_{m\epsilon}(\hat{f}^{(1)})$.
        \STATE Output the $m$-th row of $\bfF_e$: $F(m,:) =\sqrt{1-m^2\epsilon^2} \mathrm{IFFT}(\hat{f}^{(m)})$.
    \ENDFOR
    \ENSURE  $\bfF$
     \end{algorithmic}
     \end{algorithm}
{  In \textbf{Algorithm~\ref{alg3}}}, $\bff=[f_1,\cdots,f_N]^T$ is the vector of sample values of $f\in H^2$ at the set $T=\{t_1,\cdots,t_N\}$ with $t_j=e^{\frac{2j\pi i}{N}}$; $\Rfrak_r:\ell^2\to\ell^2$ is defined by \eqref{Rr}; $F(m,n) =\langle f,e_z\rangle$, where $z=(m\epsilon) e^{\frac{2n\pi i}{N}}\in \cG$; and $\bfF =[F(m,n)]_{m,n=1}^{M,N}$.

\section{{Illustrative Examples}}\label{sec4}
In this section, we illustrate the accuracy and effectiveness of our CAFD-CGD algorithm in approximating the functions in $H^2$ and in recovering the tuple of a Blachke form. We also compare it with CAFD\cite{Qian13} to demonstrate its improvement.

We run the experiments using Matlab 2017b on the laptop \emph{Dell Latitude E7440} with 8GB RAM, Intel Core i7-4600U CPU @ 2.10GHz, and Microsoft Windows 10. \par
For making fair comparisons, we sample the functions at $1024$ equidistant points on the unit circle for all experiments. The parameters for CAFD-CGD are set as follows: $t_k=1$ for all $k$ (in \eqref{eq3c}). In ITS, for function approximation, we select $\epsilon=0.01$ and $\delta=2^{-8}$, which produce $25344$ nodes on the polar mesh grid set, while for tuple recovering, we set $\delta=2^{-7}$, which reduce a half of the nodes. In CAFD, we set the grid gap to $0.01$, which produces $30752$ nodes on the rectangular mesh grid set.

For convenience, in all examples we write $\tau=e^{it}$, so that when $t$ is equidistantly sampled on $[0, 2\pi]$, $\tau$ is equidistantly sampled on the complex unit circle.
\subsection{Approximating functions in $H^2$}
In this subsection, we illustrate the accuracy of CAFD-CGD for approximating functions in $H^2$ space. The graphs only show the real part of each complex function.

\noindent\textbf{Example 5.1.} In this example, we approximate three functions: $f_1(\tau)=\frac{1}{2+\tau^4}, f_2(\tau)= e^{\tau^2}, f_3(\tau)=\log(2+\tau^2)$ using $6$-order Blaschke form. The results are shown in Figure~\ref{fig1}, and a comparison of results is in Table~\ref{tab1}. Since $f_1(\tau)$ is a rational function with the order $4$, it is not a surprise that its approximations have the highest accuracy. In all cases, CAFD-CGD performances better and costs less CPU time.
\smallskip

\noindent\textbf{Example 5.2.} In this example, we approximate three functions: $f_1(\tau)=1+\tau^2+\tau^4+ 1/(3+\tau^2), f_2(\tau)= \cos{\tau^2}$ and $f_3(\tau)=\frac{\cos 6\tau^2}{2+\tau^2}$, which have more windings than those in Example 5.1. Therefore, we choose 10-order for the first two and 30-order for the last one. The results are shown in Figure~\ref{fig2}, and a comparison is given in Table~\ref{tab2}. CAFD-CGD still shows the better performance in average and costs less time.

\begin{table}[htbp]
\centering
\caption{\small{Comparison of CAFD-CGD and CAFD in Approximation (Example 5.1)}}
        \begin{tabular}{|c|c|c|c|c|c|c|}
            \cline{1-7}
            \multicolumn{1}{|c|}{Function}  & \multicolumn{2}{|c|} {$1/(2+\tau^4)$} & \multicolumn{2}{|c|} {$e^{\tau^2}$} &  \multicolumn{2}{|c|} {$\log(2+\tau^2)$} \\
            \cline{1-7}
             { }& $L^2$ Error & CPU Time & $L^2$ Error & CPU Time  & $L^2$ Error & CPU Time \\
             \hline
        \textbf{CAFD-CGD} &0.0046\%  & 0.6102 & 0.0875\%  & 2.3746 & 0.1379\% & 0.7794 \\
            \hline
            {CAFD}       & 1.0094\%  & 5.8766 & 0.3463\% & 14.572 &  0.1922\% & 9.6438 \\
            \hline
        \end{tabular}
\label{tab1}
\end{table}

\begin{figure}[htbp]
 	\centering
 	\includegraphics[width=460pt]{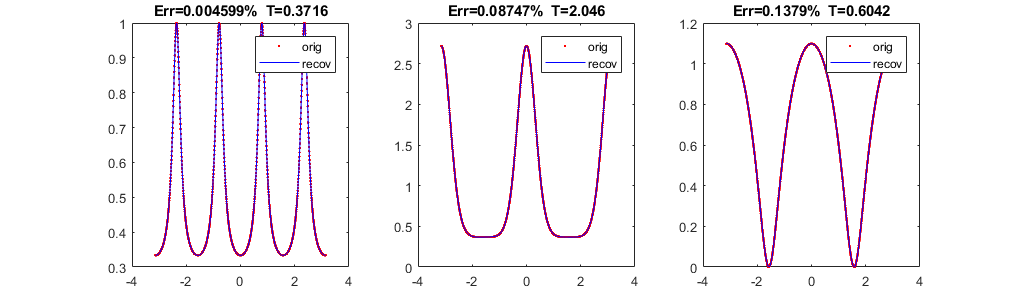}
 	\caption{\small Approximations of $\frac{1}{2+\tau^4},  e^{\tau^2}, \log(2+\tau^2)$.}
 \label{fig1}
\end{figure}

\begin{table}[htbp]
\centering
\caption{\small Comparison of CAFD-CGD and CAFD in Approximation (Example 5.2)}
        \begin{tabular}{|c|c|c|c|c|c|c|}
            \cline{1-7}
            \multicolumn{1}{|c|}{Function}  & \multicolumn{2}{|c|} {$1+\tau^2+\tau^4+ 1/(3+\tau^2)$} & \multicolumn{2}{|c|} {$\cos{\tau^2}$} &  \multicolumn{2}{|c|} {$ \cos 6\tau^2 /(2+\tau^2)$} \\
            \cline{1-7}
              & $L^2$ Error & CPU Time  & $L^2$ Error & CPU Time &  $L^2$ Error & CPU Time \\
            \hline
       \textbf{CAFD-CGD} & 0.0189\%  &1.8335 &0.0704\% & 1.8068 &  0.0134\% & 9.7661 \\
            \hline
             {CAFD} & 0.0141\%  & 26.056 & 0.0747\% & 13.939 & 0.0141\% & 33.342 \\
            \hline
        \end{tabular}
\label{tab2}
\end{table}

\begin{figure}[htbp]
 	\centering
 	\includegraphics[width=460pt]{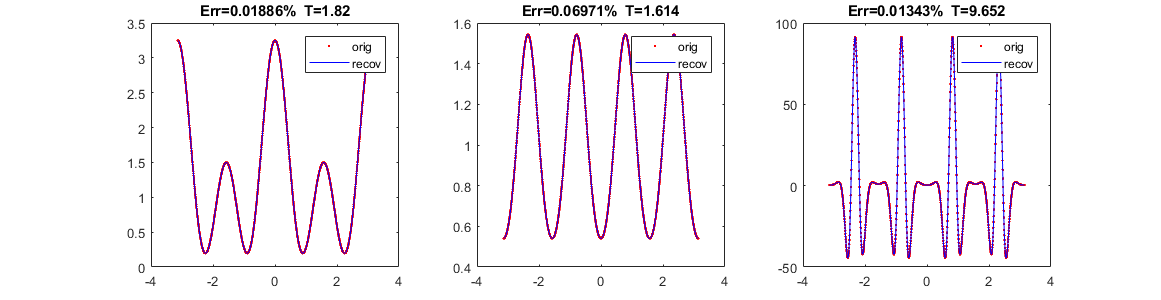}
 	\caption{\small Approximations of $1+\tau^2+\tau^4+ 1/(3+\tau^2),  \cos{\tau^2}, \frac{\cos 6\tau^2}{2+\tau^2}$.}
 \label{fig2}
\end{figure}

\subsection{Recovering Tuple of Blaschke Form}
In this subsection, we recover $n$-tuple $\bfb=[b_1,\cdots,b_n]$ of the $n$-Blaschke form
\begin{equation}\label{blk1}
f(\tau)=\sum_{k=1}^n c_kB_{b_1,\cdots,b_k} (\tau).
\end{equation}
When $\bfb$ is recovered, the coefficient vector $\bfc=[c_1,\cdots,c_n]$ in (\ref{blk1}) can be computed by (\ref{neww1}).
To recover an $n$-tuple, the approximation degree has to be $n$ too. An ideal algorithm should exactly recover $\bfb$. However, a practical algorithm hardly achieves the exact recovery because we only know the values of $f(\tau)$ on the unit circle.

We measure the $n$-tuple recovering error by the tuple distance, which is defined as follows:
Let $\cP$ be the set of all permutations on a vector $\mathbf{u}=[u_1,\cdots, u_n]$. The distance between two $n$-tuples $\mathbf{u}$ and $\mathbf{v}$ in $\bbD^n$ is
\begin{equation}\label{berror}
d(\mathbf{u}, \mathbf{v})=\min_{P\in \cP}\|P\mathbf{u}-\mathbf{v}\|.
\end{equation}
In this subsection, the graphs of complex functions show their real parts too.
\smallskip

\noindent\textbf{Example 5.3.} In this example, the function is the $5$-Blashcke form given by:\par
\noindent $\bfb$=[-0.475+0.305i, -0.180+0.715i, 0.260-0.730i, 0.540+0.360i, -0.485-0.215i].\par
\noindent $\bfc$ = [-0.5861-0.04445i, 0.2428-0.6878i, 0.4423-0.3309i, -0.2703-0.8217i, -0.8085+0.3774i].
The functions and $5$-tuples recovered by CAFD-CGD and CAFD are shown in \textbf{Figure~\ref{fig3}}. The error comparison is given in Table~\ref{tab3} and the comparison of the recovered tuple locations is given in Table~\ref{tab3-2}. The results confirm the advance of CAFD-CGD.
\smallskip

\noindent\textbf{Example 5.4}
In the second example, the function is the 4-Blaschke form created by\\
\noindent $\bfb$= [-0.4900-0.8000i, 0.3100+0.1400i, -0.9400-0.2900i, 0.2300-0.6900i].\\
\noindent $\bfc$= [1.0470+0.55587i, -0.2269-1.1203i, -0.1625-1.5327i, 0.6901-1.0979i].\\
The recovered functions and $4$-tuples are shown in \textbf{Figure~\ref{fig4}}. The error comparison is given in Table~\ref{tab4} and the comparison of the recovered tuple locations is in Table~\ref{tab4-2}. The results shows that RARL2 achieves. From the graph, we can see that the function has two peaks. Hence, the function is numerically singular at these points. Hence, the Lipschitz condition in \eqref{eqLip} for CGD fails (in the numerical sense). Hence, the gradient descent method may diverge. When we check the CAFD-CGD process carefully, we find that the algorithm stops running CGD after ITS is completed. Hence, CAFD-CGD does not work effectively when the function is not smooth enough.
\smallskip

\noindent\textbf{Example 5.5} In this example (also Example 1 in \cite{Qian13}), the function is
created by\\
\noindent$\bfb$= [0.6800+0.5200i; 0.3900+0.8100i;  -0.1300-0.8700i; 0.5500-0.1000i].\\
\noindent$\bfc$= [0.1440+0.5197i; -1.6387-0.0142i; -0.7601-1.1555i; -0.8188-0.0095i].\\
We give the recovered functions and $4$-tuples in \textbf{Figure~\ref{fig5}}. The error comparison is in Table~\ref{tab5} and the comparison of the recovered tuple locations is in Table~\ref{tab5-2}.
In this example, CAFD exactly recover the $4$-tuple because it locates on the grids of CAFD. For this example, CAFD-CGD  gives a satisfactory approximative results. Because the tuple is not on its grids, it dose not exactly recover the tuple.
\smallskip

\noindent\textbf{Example 5.6} In this example (also Example 3 in \cite{Qian13}), the function is the 4-Blaschke form created by\\
\noindent $\bfb$= [-0.1800+0.7700i, -0.0200-0.1800i, 0.1000+0.2400i, 0.1800-0.5300i].\\
\noindent $\bfc$= [0.1097+0.4754i, 1.1287+1.1741i, -0.2900+0.1269i, 1.2616-0.6568i].\\
We display the recovered functions and $4$-tuples in \textbf{Figure~\ref{fig6}}, show the error comparison in Table~\ref{tab6}, and give the comparison of the recovered tuple locations in Table~\ref{tab6-2}. The same conclusion for Example 5.3 is obtained.
\smallskip

\noindent\textbf{Example 5.7} To show the stability of CAFD-CGD, we select twenty $5$-Blaschke forms and twenty $6$-Blascke forms at random as the target functions in this experiment. We also compare it with CAFD. The comparisons for $5$-Blaschke forms and $6$-Blaschke forms are shown in Figure~\ref{fig7-1} and Figure~\ref{fig7-2}, respectively. The statistical data of the recovering is given in Table~\ref{tab7-1} and Table~\ref{tab7-2}. The results show that CAFD-CGD performs stably.

\begin{table}
\centering
\caption{\small{Comparison of recovering errors (Example 5.3)}}
        \begin{tabular}{|c|c|c|c|}
            \hline
            Algorithm  & $L^2$ Relative Error & $5$-Tuple Distance & CPU Time \\
            \hline
            \textbf{CAFD-CGD} & 0.00356\%  & 0.000177 & 0.5415 \\
            \hline
            {CAFD} &0.6813\%  &0.0304 & 41.698 \\
            \hline
        \end{tabular}
\label{tab3}
\end{table}

\begin{table}
\centering
\caption{\small{Comparison of recovered tuples (Example 5.3)}}
        \begin{tabular}{|c|c|c|c|}
            \hline
           Point  & Original  & \textbf{CAFD-CGD} & {CAFD} \\
            \hline
            1  &  -0.4850 - 0.2150i&  -0.4850 - 0.2150i & -0.5000 - 0.2150i \\
            \hline
            2 &  -0.4750 + 0.3050i & -0.4750 + 0.3049i & -0.4750 + 0.3250i \\
            \hline
            3 &  0.5400 + 0.3600i  & 0.5340 + 0.3600i &  0.5450 + 0.3550i  \\
             \hline
            4 &  -0.1800 + 0.7150i & -0.1802 + 0.7150i & -0.1650 + 0.7150i \\
            \hline
            5 & 0.2600 - 0.7300i &  0.2600 - 0.7300i &  0.2550 - 0.7300i  \\
            \hline
        \end{tabular}
\label{tab3-2}
\end{table}

\begin{figure}[htbp]
	\centering
	\subfigure{\includegraphics[width=210pt]{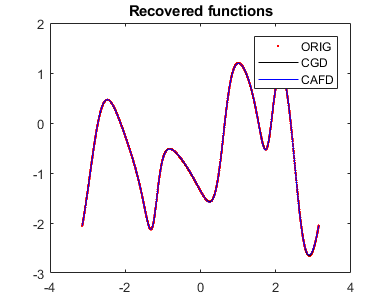}}	
	\subfigure{\includegraphics[width=210pt]{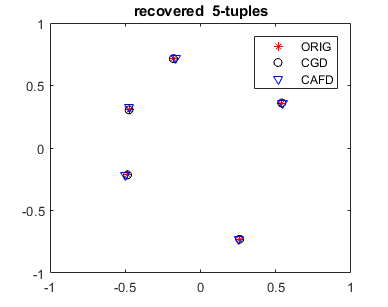}}
	\caption{\small{Recovering function and $5$-tuple in Example 5.3.}}
\label{fig3}
\end{figure}

\begin{table}
\centering
\caption{\small{Comparison of recovering errors (Example 5.4)}}
        \begin{tabular}{|c|c|c|c|}
            \hline
            Algorithm  & $L^2$ Error & $4$-Tuple Distance & CPU Time \\
            \hline
            \textbf{CAFD-CGD} & 9.7127\% & 0.0920 & 0.0891 \\
            \hline
            {CAFD} & 0.5689\%   & 0.0255 & 15.272 \\
            \hline
        \end{tabular}
\label{tab4}
\end{table}

\begin{table}
\centering
\caption{\small{Comparison of recovered tuples (Example 5.4)}}
        \begin{tabular}{|c|c|c|c|}
            \hline
           Point     & Original        & \textbf{CAFD-CGD}   & {CAFD}  \\
            \hline
            1        &  0.3100+0.1400i & 0.3880+0.0972i & 0.3300+0.1250i \\
            \hline
            2        &  0.2300-0.6900i & 0.2493-0.6967i & 0.2300-0.6950i \\
            \hline
            3        & -0.4900-0.8000i &-0.4833-0.8063i &-0.4900-0.8000i  \\
             \hline
            4        & -0.9400-0.2900i &-0.9378-0.2843i &-0.9400-0.2900i \\
            \hline
        \end{tabular}
\label{tab4-2}
\end{table}

\begin{figure}[htbp]
 	\centering
 	\subfigure{\includegraphics[width=210pt]{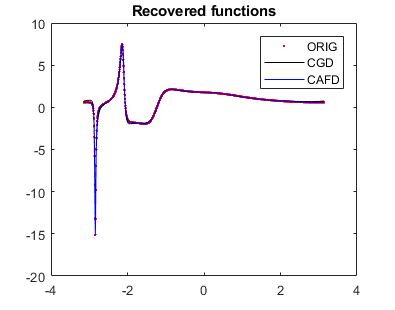}}	
 	\subfigure{\includegraphics[width=210pt]{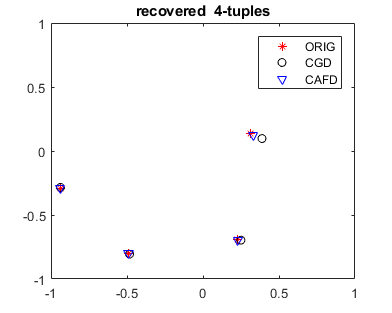}}
 	\caption{\small{Recovering the function and its $4$-tuple in Example 5.4.}}
\label{fig4}
\end{figure}

\begin{table}
\centering
\caption{\small{Comparison of recovering errors (Example 5.5)}}
        \begin{tabular}{|c|c|c|c|}
            \hline
            Algorithm  & $L^2$ Error & $4$-Tuple Distance & CPU Time \\
            \hline
            \textbf{CAFD-CGD} & 0.0052\%   & 0.0001 & 0.2850 \\
            \hline
            {CAFD} & 0.0000\% & 0.0000 & 22.45\\
            \hline
        \end{tabular}
\label{tab5}
\end{table}

\begin{table}
\centering
\caption{\small{Comparison of recovered tuples (Example 5.5)}}
        \begin{tabular}{|c|c|c|c|}
            \hline
           Point & Original        & \textbf{CAFD-CGD}   & {CAFD}   \\
            \hline
            1        &  0.5500-0.1000i & 0.5500-0.0999i & 0.5400-0.1000i \\
            \hline
            2        &  0.6800+0.5200i & 0.6800+0.5200i & 0.6800+0.5200i \\
            \hline
            3        & -0.1300-0.8700i &-0.1300-0.8700i &-0.1300-0.8700i \\
             \hline
            4        &  0.3900+0.8100i & 0.3900+0.8100i & 0.3900+0.8100i \\
            \hline
        \end{tabular}
\label{tab5-2}
\end{table}

\begin{figure}[htbp]
	\centering
	\subfigure{\includegraphics[width=210pt]{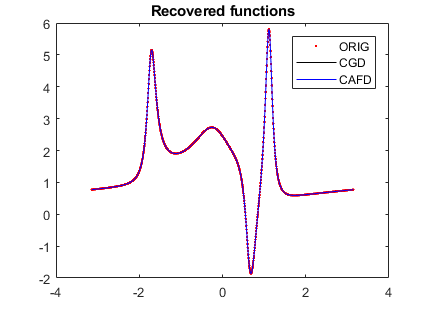}}	
	\subfigure{\includegraphics[width=210pt]{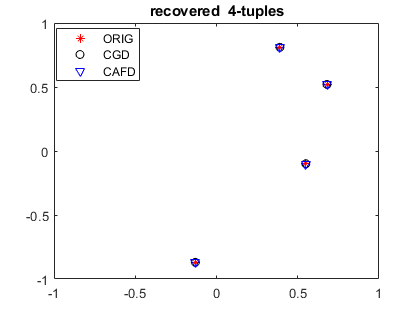}}
	\caption{\small{Recovering the function and its $4$-tuple in Example 5.5.}}
\label{fig5}
\end{figure}

\begin{table}
\centering
\caption{\small{Comparison of recovering errors (Example 5.6)}}
        \begin{tabular}{|c|c|c|c|}
            \hline
            Algorithm  & $L^2$ Error & $4$-Tuple Distance & CPU Time \\
            \hline
            \textbf{CAFD-CGD} & 0.0072\% & 0.0010 & 1.114 \\
            \hline
            {CAFD} & 2.0675\% & 0.2954 & 52.756 \\
            \hline

        \end{tabular}
\label{tab6}
\end{table}

\begin{table}
\centering
\caption{\small{Comparison of recovered tuples (Example 5.6)}}
        \begin{tabular}{|c|c|c|c|}
            \hline
           Point & Original        & \textbf{CAFD-CGD}   & {CAFD} \\
            \hline
            1        & -0.0200-0.1800i &-0.0194-0.1802i & 0.0200+0.3550i \\
            \hline
            2        &  0.1000+0.2400i & 0.0994+0.2403i & 0.0700-0.3700i \\
            \hline
            3        &  0.1800-0.5300i & 0.1799-0.5298i & 0.1700-0.3950i \\
             \hline
            4        & -0.1800+0.7700i & 0.1798+0.7703i &-0.1100+0.7850i \\
            \hline
        \end{tabular}
\label{tab6-2}
\end{table}

\begin{figure}[htbp]
	\centering
	\subfigure{\includegraphics[width=210pt]{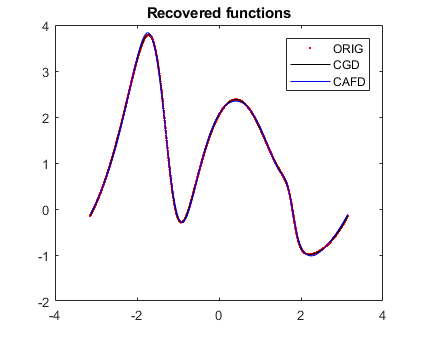}}	
	\subfigure{\includegraphics[width=210pt]{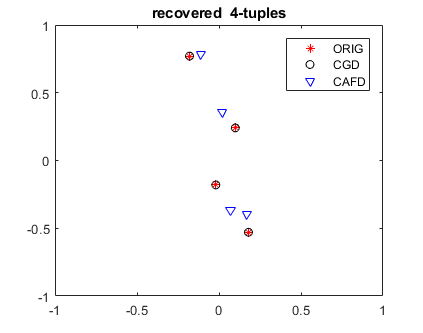}}
	\caption{\small{Recovering the function and its $4$-tuple in Example 5.6.}}
\label{fig6}
\end{figure}

\begin{figure}[htbp]
 	\centering
 	\includegraphics[width=460pt]{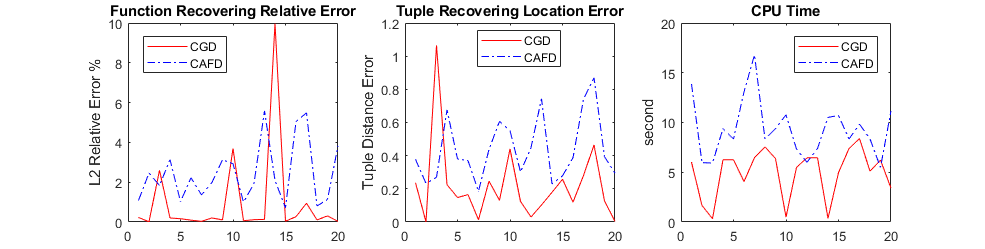}
 	\caption{\small{Recovering randomly selected 20 $5$-Blaschke forms.}}
 \label{fig7-1}
\end{figure}

\begin{figure}[htbp]
 	\centering
 	\includegraphics[width=460pt]{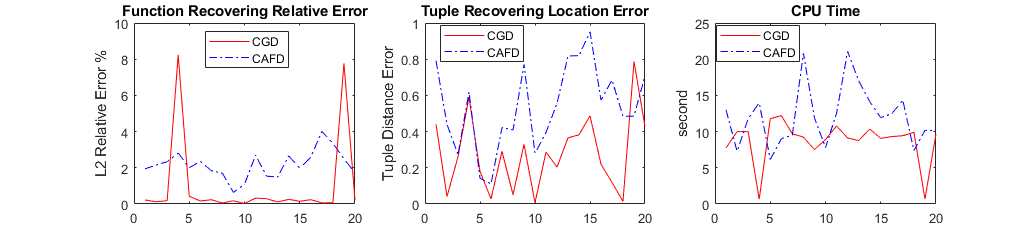}
 	\caption{\small{Recovering randomly selected twenty $6$-Blaschke forms.}}
 \label{fig7-2}
\end{figure}
\medskip

\begin{table}
\centering
\caption{\small{Statistics for recovering 20 $5$-Blaschke.}}
        \begin{tabular}{|c|c|c|c|c|c|c|c|c|c|}
        \cline{1-10}
            \multicolumn{1}{|c|}  { } & \multicolumn{3}{|c|} {$L^2$ Relative Error} & \multicolumn{3}{|c|} {Tuple Distance} & \multicolumn{3}{|c|} {CPU Time}\\
            \cline{1-10}
            {Tuple} & mean & max  & std  & mean & max  & std & mean & max  & std \\
            \hline
            \textbf{CAFD-CGD} & 0.9648\% & 9.9677\% & 2.3215 & 0.218  & 1.065 & 0.23636 & 4.0982 & 7.0615 & 2.0868  \\
            \hline
            {CAFD}       & 2.4142\% & 5.4597\% & 1.5701 & 0.4268 & 0.8560 & 0.1906 & 8.9664 & 14.068 & 2.4139 \\
            \hline
        \end{tabular}
\label{tab7-1}
\end{table}

\begin{table}
\centering
\caption{\small{Statistics for recovering twenty $6$-Blaschke forms.}}
        \begin{tabular}{|c|c|c|c|c|c|c|c|c|c|}
        \cline{1-10}
            \multicolumn{1}{|c|}  { } & \multicolumn{3}{|c|} {$L^2$ Relative Error} & \multicolumn{3}{|c|} {Tuple Distance} & \multicolumn{3}{|c|} {CPU Time}\\
            \cline{1-10}
            {Tuple} & mean & max  & std  & mean & max  & std & mean & max  & std \\
            \hline
            \textbf{CAFD-CGD} & 0.9669\% & 8.2423\% & 2.4106 & 0.2757 & 0.7870 & 0.2087 & 7.0792 & 9.7105 & 2.4354\\
            \hline
            {CAFD}       & 1.9579\% & 4.4610\% & 0.8814 & 0.5294 & 0.9382 & 0.2557 & 11.733 & 18.312 & 3.3778\\
            \hline
        \end{tabular}
\label{tab7-2}
\end{table}
\newpage

\subsection{Conclusion}
Our experiments are based on the functions with 1024 equidistant samples on the complex unit circle. The experiments show that CAFD-CGD method improves CAFD in the Blaschke form approximation. It spends less CPU time while obtains the better approximation in average. However, it is sensitive to the smoothness of target functions. When the gradient of the target function does not satisfy the Lipschitz condition, as shown in Examples 5.4 and 5.5, it can only achieve the similar result as CAFD. Particularly, in Example 5.5, the $4$-tuple is locates on the grid set of CAFD, (but not on the grid set of CAFD-CGD). Hence, CAFD finds its exact location.

CAFD is more stable because it does not utilize the gradient of a function. When the mesh gird is dense enough, or occasionally, a best tuple is very near a node of the mesh grid set, the approximation error is small. However, because its approximative accuracy is limited by the grid gap, usually it cannot reach the same accuracy as CAFD-CGD. Because the standard CAFD searches the best tuple on a rectangular mesh grid set, which cannot adopt FFT algorithm, it spends more CPU time than CAFD-CGD.

In the best tuple recovering, it seems that the closeness of the tuple has a high correlation with the accuracy of $L^2$ approximation, but it is not synchronous: A recovered tuple being closer to the target one does not always produces a better $L^2$ approximation. The relationship between them needs further investigation.

\section*{Acknowledgment}
This work was supported by University of Macau Research Grant MYRG116(Y1-L3)-FST13-QT, Macao Government FDCT 098/2012/A3, and Sam Houston State University Research Grant ERG-250711.

\end{document}